\documentclass[reqno]{amsart}

\usepackage{amssymb,amsmath,amsthm,xy,amsfonts,latexsym,booktabs,tikz,relsize}

\usetikzlibrary{arrows,automata}
\usetikzlibrary{positioning}
\tikzset{
    state/.style={
           rectangle,
           rounded corners,
           draw=black, thick,
           minimum height=2em,
           inner sep=2pt,
           text centered,
           },
}

\theoremstyle{definition}
\newtheorem{definition}{Definition}

\newtheorem{remark}[definition]{Remark}

\theoremstyle{plain}
\newtheorem{lemma}[definition]{Lemma}
\newtheorem{proposition}[definition]{Proposition}

\renewcommand{\L}{L_\mathlarger{\mathlarger{\cdot}}}
\newcommand{\las}{(\!(}
\newcommand{\ras}{)\!)}

\allowdisplaybreaks

\begin{document}

\title[Splitting of operations for alternative and Malcev structures]{Splitting of operations for alternative and Malcev structures}

\author[S.~Madariaga]{Sara Madariaga}%
\thanks{The research of the author was supported by a postdoctoral fellowship from PIMS (Pacific Institute for the Mathematical Sciences)
and the Spanish Ministerio de Ciencia e Innovaci\'on (MTM2010-18370-C04-03).}

\email{sara.madariaga@unirioja.es}

\begin{abstract}
    In this paper we define pre-Malcev algebras and alternative quadri-algebras and
    prove that they generalize pre-Lie algebras and quadri-algebras respectively to the alternative setting.
    We use the results and techniques from \cite{BaiBellierGuoNi12,GubarevKolesnikov13} to discuss and give explicit computations
    of different constructions in terms of bimodules, splitting of operations, and Rota-Baxter operators.
\end{abstract}



\maketitle


\section{Introduction}

Associative dialgebras were introduced by Loday and Pirashvili \cite{LodayPirashvili93} as the universal
enveloping algebras of Leibniz algebras (a noncommutative analogue of Lie algebras).
Loday also defined dendriform dialgebras in his study of algebraic $K$-theory \cite{Loday01}.
Dendriform dialgebras have two operations whose sum is an associative operation.
Moreover, the operads associated to associative dialgebras and dendriform dialgebras are Koszul dual.
Loday and Ronco introduced trialgebras, a generalization of dialgebras, and, dual to them, dendriform trialgebras \cite{LodayRonco04}.
This means one can decompose an associative product into sums of two or three operations satisfying some compatibility conditions.

Aguiar \cite{Aguiar00} first noticed a relation between Rota-Baxter algebras and dendriform dialgebras.
He proved that an associative algebra with a Rota-Baxter operator $R$ of weight zero is a dendriform dialgebra
relative to operations $ a \prec b = aR(b)$ and $a \succ b = R(a)b$.
Ebrahimi-Fard \cite{EbrahimiFard02} generalized this fact to Rota-Baxter algebras of arbitrary weight
and obtained dendriform dialgebras and dendriform trialgebras.
Universal enveloping Rota-Baxter algebras of weight $\lambda$ for dendriform dialgebras and trialgebras were defined by
Ebrahimi-Fard and Guo in \cite{EbrahimiFardGuo08}.

This idea of splitting associative operations gave rise to different algebraic structures with multiple binary operations
such as quadri-algebras, ennea-algebras, NS-algebras, dendriform-Nijenhuis algebras or octo-algebras.
These constructions can be put into the framework of (black square) products of nonsymmetric operads \cite{EbrahimiFardGuo05}.

Analogues of dendriform dialgebras, quadri-algebras and octo-algebras for Lie algebras \cite{BaiLiuNi10,LiuNiBai11},
Jordan algebras \cite{BaiHou12,HouNiBai13}, alternative algebras \cite{NiBai10}, or Poisson algebras \cite{Aguiar00} have been obtained.
They can be regarded as splitting the operations in these latter algebras.
Figure~\ref{fig:classic} illustrates the relations existing between some of these structures.
(Note that Rota-Baxter operators can be replaced by the more general $\mathcal{O}$-operators \cite{Kupershmidt99} in the upper and lower rows.)

\begin{figure}
\includegraphics{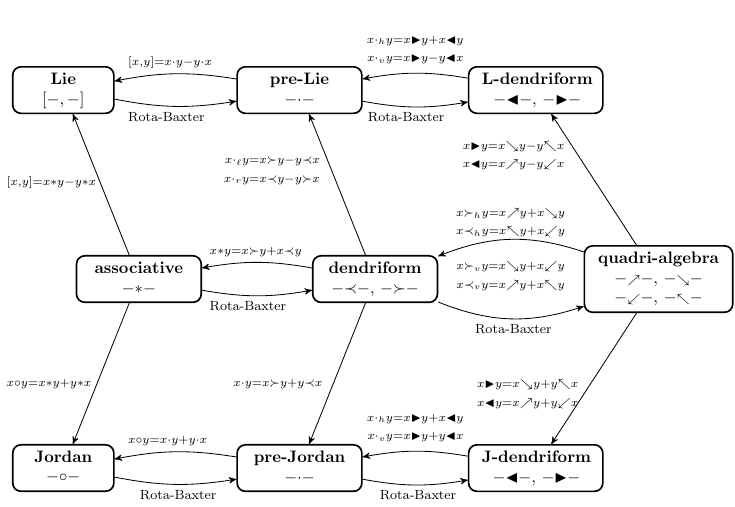}
\caption{Associative setting}
\label{fig:alternative}
\end{figure}

Bai, Bellier, Guo and Ni \cite{BaiBellierGuoNi12} set up a general framework to make precise the notion of splitting of
any binary operad and generalized the relationship of splitting of an operad with the Manin black product given by
Vallette \cite{Vallette08} and also with Rota-Baxter operators.
To do this, they constructed the disuccessor and the trisuccessor of the (non-necessarily quadratic) binary algebraic operad
governing a variety $\mathrm{Var}$ of algebras defined by generating operations and relations.
This allows the splitting of other operads and gives a general way to relate known operads and to produce new operads from them.
They showed that constructing the disuccessor (resp. trisuccessor) of a quadratic binary operad is equivalent to
computing its Manin black product with the operad $\mathrm{PreLie}$ (resp. $\mathrm{PostLie}$) and this construction can
be obtained from a Rota-Baxter operator of weight zero (resp. nonzero).
The resulting operads were named as di- (resp. tri-) $\mathrm{Var}$-dendriform algebras by Gubarev and Kolesnikov and
are the Koszul dual operads of the corresponding di- (resp. tri-) $\mathrm{Var}$-algebras \cite{GubarevKolesnikov13}.

\subsection*{Dendriform dialgebras and disuccessors of algebras}

Following the notation in \cite{GubarevKolesnikov13}, by an $\Omega$-algebra we mean a vector space
equipped with a family of bilinear operations $\Omega = \{ \circ_i \mid i \in I \}.$
We denote by $\mathcal{F}$ the free operad governing the variety of all $\Omega$-algebras.
For every natural number $n > 1$, the space $\mathcal{F}(n)$ can be identified with
the space spanned by all binary trees with $n$ leaves labeled by $x_1, \dots , x_n$,
where each vertex (which is not a leaf) has a label from $\Omega$.
Let $\mathrm{Var}$ be a variety of $\Omega$-algebras defined by a family $S$ of multilinear identities
of any degree greater than one and denote by $\mathcal{P}_{\mathrm{Var}}$ the binary operad
governing the variety $\mathrm{Var}$, i.e., every algebra from $\mathrm{Var}$ is a functor from
$\mathcal{P}_{\mathrm{Var}}$ to the category $\mathrm{Vec}$ of vector spaces with multilinear maps.
Denote by $\Omega^{(2)}$ the set of binary operations $ \{ \prec_i, \succ_i \mid i \in I \} $ and let
$\mathcal{F}^{(2)}$ stand for the free operad governing the varieties of all $\Omega^{(2)}$-algebras.

Let $(R)$ be the operadic ideal of relators defining the operad $\mathcal{P}_{\mathrm{Var}}$
(we mean by this that $\mathcal{P}_{\mathrm{Var}} = \mathcal{F} / (R)$).
An element of $r \in (R)$ is of the form
\[
    r = \sum_\phi \phi(x_1,\dots,x_n) = 0,
\]
where $\phi$ is a composite of operations from $\Omega$. Equivalently, by the identification above, $r$ can be
regarded as a linear combination of the corresponding binary trees. The disuccessor of the operad $\mathcal{P}_{\mathrm{Var}}$
is constructed as $\mathcal{F}^{(2)}/(R)'$, where $(R)'$ is obtained applying the rules in \cite[Prop.~2.4]{BaiBellierGuoNi12}
to each generator of $(R)$ as follows:
for each $k=1, \dots, n$ we replace each occurrence of the operation (vertex label) $\circ_i$ by
\begin{itemize}
    \item $\prec_i$ if the variable (leaf label) $x_k$ appears in the left factor of $\circ_i$
          (the path from the root to the leaf labeled by $x_k$ turns left at this vertex)
    \item $\succ_i$ if the variable (leaf label) $x_k$ appears in the right factor of $\circ_i$
          (the path from the root to the leaf labeled by $x_k$ turns right at this vertex)
    \item $\prec_i + \succ_i$ if the variable (leaf label) $x_k$ does not appear in either the left or the right factor of $\circ_i$
          (the path from the root to the leaf labeled by $x_k$ does not pass through this vertex).
\end{itemize}
This is equivalent to compute the Manin black product $\mathcal{P} \bullet \mathrm{PreLie}$ \cite[Thm.~3.2]{BaiBellierGuoNi12}
(For the precise definition and properties of this construction as well as of the trisuccessor, see \cite{BaiBellierGuoNi12}.)

Quoting \cite{GubarevKolesnikov13}, for a binary operad $\mathcal{P}$ the disuccessor procedure described above gives rise to
what is natural to call identities defining di-$\mathcal{P}$-dendriform algebras and can be easily generalized for algebras with
$n$-ary operations, $n \ge 2$.
If $\mathcal{P}$ is quadratic, then these di-$\mathcal{P}$-algebras are Koszul dual to the corresponding di-$\mathcal{P}^!$-algebras
since $(\mathcal{P}^! \otimes \mathrm{Perm})^! = \mathcal{P} \bullet \mathrm{PreLie}$.

\subsection*{Rota-Baxter operators and their generalizations}

A Rota-Baxter operator in a binary algebra $A$ (with product denoted by juxtaposition) is a linear map $R\colon A \to A$ satisfying
\[
    R(x)R(y) = R\big( R(x)y + xR(y) + \lambda xy \big) \qquad \forall x,y \in A,
\]
where $\lambda$ is a fixed scalar. Associative Rota-Baxter algebras first appeared in 1960 in the works by Baxter on probability.
They became popular in different areas of mathematics and physics and recently they have been used in areas such as quantum field theory,
Yang-Baxter equations, shuffle products, operads, Hopf algebras, combinatorics or number theory (see \cite{EbrahimiFardGuo08} and the references therein).
Rota-Baxter algebras allow us to define dendriform dialgebra ($\lambda=0$, see \cite{Aguiar00}) and trialgebra structures
($\lambda \ne 0$, see \cite{EbrahimiFard02}) in associative algebras.
Moreover, one can construct universal enveloping Rota-Baxter algebras of dendriform di and trialgebras \cite{EbrahimiFardGuo08}.

An interesting question is whether we can recover any dendriform di or trialgebra from a Rota-Baxter algebra structure in the same underlying vector space,
not as a dendriform di or tri subalgebra of the corresponding universal enveloping Rota-Baxter algebra.
A counterexample is given by considering the set of Rota-Baxter operators defined in $V = \mathbb{C} e_1 \oplus \mathbb{C} e_2$ \cite{LiHouBai07};
they produce six non-isomorphic dendriform algebra structures in $V$ but, according to \cite{Zhang08}, there are at least five more dendriform algebra
structures which can be defined in $V$.
To give a positive answer to this problem, Bai, Guo and Ni \cite{BaiGuoNi13} introduced the notion of relative Rota-Baxter operators,
a generalization of both Rota-Baxter operators and $\mathcal{O}$-operators
which allows us to establish bijections between equivalent classes of relative Rota-Baxter operators and dendriform di and trialgebras.
This generalizes the work by Uchino \cite{Uchino08} for dendriform dialgebras.

In this paper we use Rota-Baxter operators to split operations, although a generalization exists in the alternative setting in terms of bimodules:
the $Al$-operators defined by Ni and Bai \cite{NiBai10}. Figure~\ref{fig:alternative} summarizes the results of the present work.

\begin{figure}
\includegraphics{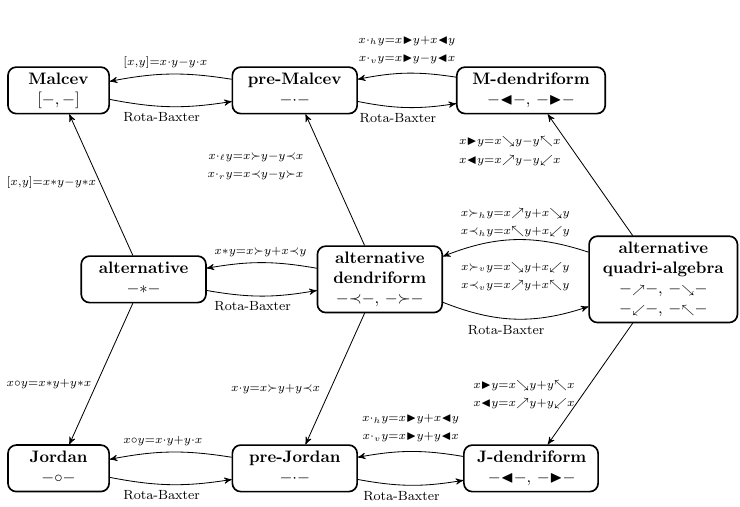}
\caption{Alternative setting}
\label{fig:alternative}
\end{figure}

\subsection*{Nonassociative polynomials and computational methods}

We use computer algebra to prove some of the results in this paper related to multilinear polynomial identities.
These computations were performed with Maple 17 and the algorithms were implemented by the author.
Here is a summary of the concepts and techniques involved.

Let $\Omega$ be a set of multilinear operations defined in a vector space $A$.
By a polynomial identity of degree $d$ satisfied by $A$ we mean an element $f \in F\{ \Omega; X \}$ of the free multioperator algebra
over $\Omega$ with set $X$ of generators such that $f(a_1, \dots, a_d) = 0$ for all $a_1, \dots, a_d \in A$.
We denote by $A_d$ the set of multilinear polynomial identities of degree $d$ satisfied by $A$.
If we assume that the base field $F$ has characteristic zero, then any polynomial identity of degree $d$ is equivalent to a finite subset of $A_d$.
The association types of degree $d$ for $\Omega$ are all possible compositions of operations from $\Omega$ involving $d$ arguments.
The multilinear monomials of degree $d$ for $\Omega$ are obtained by substituting all possible permutations of
$d$ variables for the $d$ arguments of the association types of degree $d$.

If we have a multioperator algebra $(A,\Omega_A)$ and the free algebra $F_\mathcal{V}\{X\}$ in the variety $\mathcal{V}$ of algebras
then we can identify an $n$-ary operation $\omega_A \in \Omega_A$ with a multilinear polynomial $p_{\omega_A} \in F_\mathcal{V}\{X\}_n$.
For each degree $d$ we consider a linear map $E_d \colon A_d \to F_\mathcal{V}\{X\}_d$, which we call the expansion map,
defined by replacing each occurrence of $\omega_A$ by $p_{\omega_A}$ and making the appropriate substitutions.
For example, if $(A,[-,-])$ is a Lie algebra and $F_\mathcal{V}\{X\}$ is the free associative algebra with product $\ast$,
then we can identify $[x,y]$ with $x \ast y - y \ast x$ so that
    \[
        E_3( [[x,y],z ] ) = (x \ast y) \ast z - (y \ast x) \ast z - z \ast (x \ast y) + z \ast (y \ast x).
    \]
The kernel of $E_d$ consists of the polynomial identities of degree $d$ satisfied by the multilinear operations
defined by the elements $\{ p_{\omega_A} \mid \omega_A \in \Omega_A \}$ in $F_\mathcal{V}\{X\}_d$.

A multilinear polynomial $f = f(x_1,\dots,x_d) \in F\{ \Omega; X \}$ can be lifted to higher degrees using the operations $\omega \in \Omega$.
To lift $f$ using an operation $\omega$ of arity $k$ we need to introduce $k{-}1$ new variables so the degree of the liftings is $d+k-1$.
For each $i = 1, \dots, d$ we replace the $i$-th variable in $f$ by $\omega$, obtaining the lifting $f(x_1,\dots,\omega(x_i,x_{d+1},\cdots,x_{d+k-1}),\dots,x_d)$.
For each $j = 1, \dots, k$ we replace the $j$-th argument of $\omega$ by $f$, obtaining the lifting
$\omega(x_{d+1},\dots,f(x_1,\dots,x_d),\dots,x_{d+k-1})$.
To lift $f$ to a given higher degree $D$ we need to consider all possible combinations of liftings which produce a result of degree $D$.

For each degree $d$ the vector space $A_d$ of multilinear polynomials has the structure of an $S_d$-module,
the action given by permuting the subscripts of the variables. The set of monomials of degree $d$ is a linear basis of $A_d$.
The multilinear identities in degree $d$ which are consequences (linear combinations of permutations of the variables)
of the liftings of identities of lower degrees form a submodule $L_d$ of $A_d$.

Throughout this paper we will compute kernels of expansion maps to find certain modules of multilinear polynomial identities
and compare them with submodules arising from liftings of identities in lower degrees.
We also compute the module generators for these subspaces of multilinear identities as well as for the quotient module
$N_d = A_d/L_d$ of new identities in degree $d$.

For a more detailed discussion about the computational methods see \cite{BremnerMurakamiShestakov07}.


\section{Pre-Malcev algebras} \label{sec:preMalcev}

In this section we define pre-Malcev algebras and give different constructions for them in the spirit of \cite{HouNiBai13}.

\begin{definition}[Malcev \cite{Malcev55}]
    A \emph{Malcev algebra} is a vector space $M$ endowed with an anticommutative bilinear product $[-,-]$ satisfying the Malcev identity
    \[
        J( x,y,[x,z] ) = [ J(x,y,z), x ]
    \]
    for all $x,y,z \in M$, where $J(x,y,z) = [[x,y],z] + [[y,z],x] + [[z,x],y]$ is the Jacobian of $x,y,z$.
\end{definition}

\begin{remark}
    If we work over a field of characteristic not 2 then the Malcev identity is equivalent to Sagle's identity:
    \[
        [[x,z],[y,t]] = [[[x,y],z],t] + [[[y,z],t],x] + [[[z,t],x],y] + [[[t,x],y],z].
    \]
    Lie algebras are examples of Malcev algebras.
    Alternative algebras are Malcev admissible algebras, i.e.,
    if $(A,\ast)$ is an alternative algebra then the anticommutator $[x,y] = x \ast y - y \ast x$ defines a Malcev structure in $A$.
    These Malcev algebras are called special.
    An important example of a (special) Malcev algebra is the set of zero-trace octonions.
\end{remark}

\begin{definition}[Kuzmin \cite{Kuzmin68}]
    A \emph{representation} of a Malcev algebra $(M,[-,-])$ on a vector space $V$ is a map $\rho \colon M \to \mathrm{End}(V)$ such that
    \[
        \rho([[x,y],z]) = \rho(x)\rho(y)\rho(z) - \rho(z)\rho(x)\rho(y) + \rho(y)\rho([z,x]) - \rho([y,z])\rho(x)
    \]
    for all $x,y,z \in M$.
\end{definition}

\begin{definition}
    A \emph{pre-Malcev algebra} is a vector space $A$ endowed with a bilinear product $\cdot$ satisfying the identity
    $PM(x,y,z,t) =0$, where
    \begin{align*}
          & PM(x,y,z,t) = (y \cdot z) \cdot (x \cdot t) - (z \cdot y) \cdot (x \cdot t)
          \\
          &
            + ((x \cdot y) \cdot z) \cdot t - ((y \cdot x) \cdot z) \cdot t
            - (z \cdot (x \cdot y)) \cdot t + (z \cdot (y \cdot x)) \cdot t
          \\
          &
            + y \cdot ((x \cdot z) \cdot t) - y \cdot ((z \cdot x) \cdot t)
            - x \cdot (y \cdot (z \cdot t)) + z \cdot (x \cdot (y \cdot t))
    \end{align*}
    for all $x,y,z,t \in A$.
\end{definition}

\begin{remark}
    This definition of pre-Malcev algebras comes from applying the splitting procedure described in \cite{BaiBellierGuoNi12}
    to the anticommutative and Malcev's identities.
    $PM(x,y,z,t)$ is the sole generator of the S4-module generated by the obtained identities
    (we applied the module generators algorithm detailed in \cite{BremnerMadariagaPeresi14}). \\[-4pt]

    \noindent Although some of the next results follow from \cite{BaiBellierGuoNi12,GubarevKolesnikov13},
    we decided to include some of the computations for completeness.
\end{remark}

\begin{proposition} \label{prop:preMalcev}
    Let $(A,\cdot)$ be a pre-Malcev algebra.
    \begin{enumerate}
        \item The commutator $[x,y] = x \cdot y - y \cdot x$ defines a Malcev algebra $\mathfrak{M}(A)$.
                In particular, pre-Malcev algebras are Malcev admissible.
        \item The left multiplication operator $\L$
        \footnote{
        In a binary algebra $(A,\cdot)$, the operator $\L_x$ of left multiplication by an element $x \in A$ is defined as
        $ \L_x \colon A \rightarrow A$, $y \mapsto x \cdot y$.
        }
        gives a representation of the Malcev algebra $\mathfrak{M}(A)$, that is,
                \[
                    \L_{[[x,y],z]} = \L_{x}\L_{y}\L_{z} - \L_{z}\L_{x}\L_{y} + \L_{y}\L_{[z,x]} - \L_{[y,z]}\L_{x}
                \]
                for all $x,y,z \in A$.
    \end{enumerate}
\end{proposition}

\begin{proof}
    Part (1) is a consequence of the definition of pre-Malcev algebras and the properties of splitting of operations \cite{BaiBellierGuoNi12}.
    For part (2) we compute
    \begin{align*}
        & \left( \L_{[[x,y]z]} - \L_{x}\L_{y}\L_{z} + \L_{z}\L_{x}\L_{y} - \L_{y}\L_{[z,x]} + \L_{[y,z]}\L_{x} \right) (t)
        \\
        &=
        [[x,y],z] \cdot t - x \cdot(y \cdot (z \cdot t)) + z \cdot (x \cdot (y \cdot t)) - y \cdot ([z,x] \cdot t) + [y,z] \cdot (x \cdot t)
        \\
        &=
        \big( (x \cdot y ) \cdot z - (y \cdot x) \cdot z - z \cdot (x \cdot y) + z \cdot (y \cdot x) \big) \cdot t - x \cdot(y \cdot (z \cdot t))
        \\
        & \quad
        + z \cdot (x \cdot (y \cdot t)) - y \cdot \big((z \cdot x - x \cdot z) \cdot t \big) + (y \cdot z - z \cdot y) \cdot (x \cdot t)
        \\
        &=
        PM(x,y,z,t) = 0.
    \end{align*}
\end{proof}

\begin{proposition}
    Let $A$ be a vector space with a binary operation $\mathlarger{\mathlarger{\cdot}}$.
    Then $(A, - \cdot -)$ is a pre-Malcev algebra if and only if
    $(A,[-,-])$ with the commutator $[x,y] = x \cdot y - y \cdot x$ is a Malcev algebra
    and $\L$ is a representation of $(A,[-,-])$.
\end{proposition}

\begin{proof}
    It follows from Prop.~\ref{prop:preMalcev} above and the definitions of Malcev algebra and representation of a Malcev algebra.
\end{proof}

\begin{remark}
    Pre-Malcev algebras generalize pre-Lie algebras. A \emph{pre-Lie algebra} \cite{Burde06,Manchon11} is a vector space $A$ with a bilinear product
    $\cdot$ satisfying the left-symmetric identity
    \[
        PL(x,y,z) = (x,y,z) - (y,x,z) = 0,
    \]
    for all $x,y,z \in A$, where $x,y,z) = (x \cdot y) \cdot z - x \cdot (y \cdot z)$ is the associator.
    Note that we can write
    \begin{align*}
        PM(x,y,z,t)
        &=
        PL( [x,y],z,t ) - PL( [y,x],z,t )+ PL( x,y,[z,t] )
        \\
        &
        + PL( y,z,[x,t] ) - [z, PL(x,y,t)] + [y, PL(x,z,t) ],
    \end{align*}
    so every pre-Lie algebra is a pre-Malcev algebra.
    The class of pre-Lie algebras is probably the most important one among all other known pre-algebras
    and was independently introduced by Vinberg, Koszul, and Gerstenhaber.
\end{remark}

Examples of pre-Malcev algebras can be constructed from Malcev algebras with Rota-Baxter operators of weight zero.

\begin{remark}
    A \emph{Rota-Baxter operator of weight zero} on a Malcev algebra $(M,[-,-])$ is a linear map $R \colon M \to M$ such that
    \[
        [ R(x), R(y) ] = R\big( [ R(x), y ] + [ x, R(y) ] \big) \qquad \forall x,y \in M.
    \]
\end{remark}

\begin{proposition} \label{prop:Malcev-RB}
    If $R \colon M \to M$ is a Rota-Baxter operator on a Malcev algebra $(M,[-,-])$ then
    there exists a pre-Malcev algebraic structure on $M$ given by
    \[
        x \cdot y = [ R(x),y ] \qquad \forall x,y \in M.
    \]
\end{proposition}

\begin{proof}
    This result follows from the properties of splitting of operations \cite{BaiBellierGuoNi12}, of which pre-Malcev algebras are a particular case.
\end{proof}

Pre-Malcev algebras are related to alternative dendriform dialgebras analogously to how pre-Lie algebras are related to dendriform dialgebras.

\begin{definition}[Ni and Bai, \cite{NiBai10}]
    An \emph{alternative dendriform dialgebra}
    \footnote{
        In the original definition by Ni and Bai this object is called pre-alternative algebra.
        However, the author prefers to use this new terminology following \cite{GubarevKolesnikov13}
        for varieties of algebras obtained applying the disuccessor procedure described in \cite{BaiBellierGuoNi12}.}
    is a vector space $A$ endowed with two bilinear products $\prec,\succ \colon A \to A$ satisfying the identities
    \begin{align*}
        (x,y,z)_m + (y,x,z)_r &= 0 &
        (x,y,z)_m + (x,z,y)_\ell &= 0 &
        \\
        (x,y,z)_r + (x,z,y)_r &= 0 &
        (x,y,z)_\ell + (y,x,z)_\ell &= 0 &
    \end{align*}
    for all $x,y,z \in A$, where
    \begin{align*}
        (x,y,z)_r &= (x \prec y) \prec z - x \prec (y \prec z + y \succ z) & \text{(right associator)} \\
        (x,y,z)_m &= (x \succ y) \prec z - x \succ (y \prec z) & \text{(middle associator)} \\
        (x,y,z)_\ell &= (x \prec y + x \succ y) \succ z - x \succ (y \succ z) & \text{(left associator)}
    \end{align*}
\end{definition}

\begin{remark}
    The variety of alternative dendriform dialgebras is obtained from the variety of alternative algebras
    by applying the disuccessor procedure described in \cite{BaiBellierGuoNi12}.
    The product $x \ast y = x \prec y + x \succ y$ in an alternative dendriform dialgebra $A$ gives an alternative
    algebra structure on $A$ called the associated alternative algebra of $A$.
    This set of identities defining alternative dendriform dialgebras is redundant:
    any of the identities in the lower row can be deduced from the other one and the identities in the upper row.
    Alternative dendriform dialgebras generalize dendriform dialgebras since the right, middle and left associators
    are identically zero in dendriform dialgebras.
\end{remark}

\begin{proposition} \label{prop:computations}
    Let $(A,\prec,\succ)$ an alternative dendriform dialgebra. Then the dendriform commutator
    \[
        x \cdot y = x \succ y - y \prec x
    \]
    defines a pre-Malcev structure in $A$. Moreover, all the identities of degree 4 satisfied by the dendriform commutator
    in the free alternative dendriform dialgebra are consequences of $PM(x,y,z,t)$.
\end{proposition}

\begin{proof}
    We consider the expansion map $E_4$ from the free binary nonassociative algebra $\mathsf{B}$
    to the free alternative dendriform dialgebra $\mathsf{BB}$.

    There are 5 association types in degree 4 for the free binary nonassociative algebra:
    \[
        ( ( - \cdot - ) \cdot - ) \cdot -
        \quad
        ( - \cdot ( - \cdot - ) ) \cdot -
        \quad
        ( - \cdot - ) \cdot ( - \cdot - )
        \quad
        - \cdot ( ( - \cdot - ) \cdot - )
        \quad
        - \cdot ( - \cdot ( - \cdot - ) )
    \]
    so there are 120 multilinear $\mathsf{B}$-monomials (apply the 24 permutations of 4 variables to the 5 association types).
    There are 40 association types in degree 4 for the free algebra with two binary operations,
    so there are 960 multilinear $\mathsf{BB}$-monomials.

    We represent $E_4$ by the $720 \times 1080$ block matrix
    \[
        E_4 = \left[ \begin{array}{c|c} A & Z \\ \midrule E & I \end{array} \right],
    \]
    where the columns of the left blocks are labeled by the $\mathsf{BB}$-monomials
    and the columns of the right blocks are labeled by the $\mathsf{B}$-monomials.
    The rows of $A$ contain the coefficient vectors of the permutations of the 25 generators
    of the module of liftings of the identities defining alternative dendriform dialgebras to degree 4.
    The rows of $E$ contain the coefficient vectors of the expansions of the nonassociative binary multilinear monomials
    of degree 4 into the free $\mathsf{BB}$-algebra using the dendriform commutator.
    $Z$ is a zero $120 \times 960$ block and $I$ is the identity matrix of size 120.

    We compute the row canonical form of this matrix, and identify the rows whose leading 1s are in the right part of the matrix.
    These 20 rows are the coefficient vectors of a set of (possibly redundant) generators for the module of identities
    satisfied by the dendriform commutator in the free alternative dendriform dialgebra.
    We find that the minimal subset of generators contains only one element, which is exactly $PM(x,y,z,t)$.
    (We note that same procedure applied to the expansion map $E_3$ does not give any identities in degree 3 for the dendriform commutator
    in the free alternative dendriform dialgebra.)
\end{proof}

\begin{remark}
    We can obtain an explicit formula of $E_4(PM(x,y,z,t))$ in terms of the linear basis of the submodule of liftings
    of the alternative dendriform dialgebra identities to degree 4.
    We construct a $960 \times 553$ matrix; in columns 1–552 we put the coefficient vectors of the basic identities
    and in column 553 we put the coefficient vector of $E_4(PM(x,y,z,t))$ (with respect to the basis of dendriform multilinear monomials).
    The last column of the row canonical form of this matrix contains the coefficients of $E_4(PM(x,y,z,t))$ with respect to
    the basic identities. This expression for $E_4(PM(x,y,z,t))$ has 109 terms.
\end{remark}


\section{Alternative quadri-algebras}

We apply the splitting of operations described in \cite{BaiBellierGuoNi12} to the variety of alternative dendriform dialgebras
to obtain a variety of algebras generalizing the quadri-algebras of Aguiar and Loday \cite{AguiarLoday04}.

\begin{definition}\label{def:altquad}
    An \emph{alternative quadri-algebra}
    is a vector space $A$ endowed with four bilinear products $\nearrow, \searrow, \swarrow, \nwarrow \colon A \to A$
    satisfying the identities
    \begin{alignat*}{2}
        \las x,y,z \ras_r + \las y,x,z \ras_m = 0 &
        \qquad \qquad
        \las x,y,z \ras_r + \las x,z,y \ras_r = 0
        \\
        \las x,y,z \ras_n + \las y,x,z \ras_w = 0 &
        \qquad \qquad
        \las x,y,z \ras_n + \las x,z,y \ras_{ne} = 0
        \\
        \las x,y,z \ras_{ne} + \las y,x,z \ras_e = 0 &
        \qquad \qquad
        \las x,y,z \ras_w + \las x,z,y \ras_{sw} = 0
        \\
        \las x,y,z \ras_{sw} + \las y,x,z \ras_s = 0 &
        \qquad \qquad
        \las x,y,z \ras_m + \las x,z,y \ras_\ell = 0
        \\
        \las x,y,z \ras_\ell + \las y,x,z \ras_\ell = 0 &
    \end{alignat*}
    for all $x,y,z \in A$, where
    \begin{align*}
        \las x,y,z \ras_r &= ( x \nwarrow y ) \nwarrow z - x \nwarrow ( y * z ) & \text{(right associator)} \\
        \las x,y,z \ras_\ell &= ( x * y ) \searrow z - x \searrow ( y \searrow z ) & \text{(left associator)} \\
        \las x,y,z \ras_{ne} &= ( x \wedge y ) \nearrow z - x \nearrow ( y \succ z ) & \text{(north-east associator)} \\
        \las x,y,z \ras_{sw} &= ( x \prec y ) \swarrow z - x \swarrow ( y \vee z ) & \text{(south-west associator)} \\
        \las x,y,z \ras_n &= ( x \nearrow y ) \nwarrow z - x \nearrow ( y \prec z ) & \text{(north associator)} \\
        \las x,y,z \ras_w &= ( x \swarrow y ) \nwarrow z - x \swarrow ( y \wedge z) & \text{(west associator)} \\
        \las x,y,z \ras_s &= ( x \succ y ) \swarrow z - x \searrow ( y \swarrow z ) & \text{(south associator)} \\
        \las x,y,z \ras_e &= ( x \vee y ) \nearrow z - x \searrow ( y \nearrow z ) & \text{(east associator)} \\
        \las x,y,z \ras_m &= ( x \searrow y ) \nwarrow z - x \searrow ( y \nwarrow z )  & \text{(middle associator)}
    \end{align*}
    \begin{alignat*}{2}
    x \succ y = x \nearrow y + x \searrow y & \qquad x \prec y = x \nwarrow y + x \swarrow y \\*
    x \vee y = x \searrow y + x \swarrow y & \qquad x \wedge y = x \nearrow y + x \nwarrow y
    \end{alignat*}
    and
    \begin{align*}
        x * y  & \; = \; x \succ y + x \prec y \; = \; x \searrow y + x \nearrow y + x \nwarrow y + x \swarrow y.
    \end{align*}
\end{definition}

\begin{remark}
    Note that in every quadri-algebra all these associators are identically zero
    (their expressions are obtained rewriting the identities defining quadri-algebras)
    so alternative quadri-algebras generalize quadri-algebras.
    These nine identities are independent (there are not redundancies).
\end{remark}

\begin{lemma}
    Let $(A,\nearrow, \searrow, \swarrow, \nwarrow)$ be an alternative quadri-algebra. Then
    $(A,\prec,\succ)$ and $(A,\vee,\wedge)$ are alternative dendriform dialgebras
    (called respectively horizontal and vertical alternative dendriform structures associated to $A$)
    and $(A,\ast)$ is an alternative algebra.
\end{lemma}

\begin{proof}
    This follows from \cite{BaiBellierGuoNi12,GubarevKolesnikov13}, as it is a particular case of splitting of operations.
    Explicitly, for $(A,\succ,\prec)$ we have that
    \begin{align*}
        &(x,y,z)_r + (y,x,z)_m
        \\
        &=
        \las x,y,z \ras_r + \las x,y,z \ras_{sw} + \las x,y,z \ras_w
        \las y,x,z \ras_s + \las y,x,z \ras_n + \las y,x,z \ras_m \\
        &(x,y,z)_r + (y,x,z)_m
        \\
        &(x,y,z)_r + (y,x,z)_r
        \\
        &=
        \las x,y,z \ras_r + \las x,y,z \ras_{sw} + \las x,y,z \ras_w
        \las x,z,y \ras_r + \las x,z,y \ras_w + \las x,z,y \ras_{sw} \\
        &(x,y,z)_m + (y,x,z)_l
        \\
        &=
        \las x,y,z \ras_s + \las x,y,z \ras_m + \las x,y,z \ras_n
        \las x,z,y \ras_\ell + \las x,z,y \ras_{ne} + \las x,z,y \ras_e,
    \end{align*}
    so $(A,\succ,\prec)$ is an alternative dendriform dialgebra. Similarly for $(A,\vee,\wedge)$.
    $(A,\ast)$ is the associated alternative algebra to both $(A,\succ,\prec)$ and $(A,\vee,\wedge)$.
\end{proof}

\begin{remark} \label{remark:new}
    A natural question arises: can any alternative algebra be obtained from alternative quadri-algebras using the construction above?
    The answer is negative: computations similar to those of Proposition \ref{prop:computations} show that
    the S3-module generated by the (linearization of the) alternative identity is a 3-dimensional S3-submodule of
    the module of identities satisfied by the operation $*$ in a dendriform quadri-algebra (which has dimension 5).
\end{remark}

Analogously to what happens for quadri-algebras \cite{AguiarLoday04}, Rota-Baxter operators allow different constructions
for alternative quadri-algebras.

\begin{remark}
    A \emph{Rota-Baxter operator on an alternative algebra} $(A,\ast)$ is a linear map $R \colon A \to A$ such that
    \[
        R(x) \ast R(y) = R \left( R(x) \ast y + x \ast R(y) \right), \\
    \]
    for all $x,y, \in A$.
    Rota-Baxter operators on alternative algebras are a particular case of $Al$-ope\-ra\-tors, defined in \cite{NiBai10}.
    A construction of alternative dendriform dialgebras from alternative algebras and $Al$-operators is given in \cite[Prop.~2.11]{NiBai10}.
    In the particular, we obtain an alternative dendriform dialgebra $A$ by defining the operations
    \[
        x \prec_R y = x \ast R(y), \quad x \succ_R y = R(x) \ast y
    \]
    in an alternative algebra $(A,\ast)$ with a Rota-Baxter operator $R$.
\end{remark}

\begin{remark}
    A \emph{Rota-Baxter operator on an alternative dendriform dialgebra} $A$ is a linear map $R \colon A \to A$ such that
    \begin{align*}
       & R(x) \succ R(y) = R \big( R(x) \succ y + x \succ R(y) \big), \\
       & R(x) \prec R(y) = R \big( R(x) \prec y + x \prec R(y) \big)
    \end{align*}
    for all $x,y, \in A$.
    Adding these two equations we get that a Rota-Baxter operator in an alternative dendriform dialgebra $A$ is also
    a Rota-Baxter operator on the associated alternative algebra.
\end{remark}

\begin{proposition} \label{prop:dendritoquadri}
    Let $(A\prec,\succ)$ an alternative dendriform dialgebra and $R \colon A \to A$ a Rota-Baxter operator on $A$.
    Then $(A,\nearrow, \searrow, \swarrow, \nwarrow)$ is an alternative quadri-algebra with the operations
    \begin{align*}
        x {\nearrow_{_R}} y = x \succ R(y)
        \qquad \qquad
        x {\searrow_{_R}} y = R(x) \succ y
        \\
        x {\swarrow_{_R}} y = R(x) \prec y
        \qquad \qquad
        x {\nwarrow_{_R}} y = x \prec R(y)
    \end{align*}
\end{proposition}

\begin{proof}
    This result follows from \cite{BaiBellierGuoNi12}, as alternative quadrialgebras are a particular case of splitting of operations.
\end{proof}

\begin{remark}
    The Rota-Baxter operator $R$ is a homomorphism of alternative dendriform dialgebras between $A$ and
    the horizontal alternative dendriform structure associated to the alternative quadri-algebra constructed with $R$:
    \begin{align*}
        R(x \prec y)
        &=
        R( x {\swarrow_{_R}} y + x {\nwarrow_{_R}} y )
        =
        R \big( R(x) \prec_R y + x \prec_R R(y) \big)
        =
        R(x) \prec_{_R} R(y) \\
        R(x \succ y)
        &=
        R( x {\nearrow_{_R}} y + x {\searrow_{_R}} y )
        =
        R \big( R(x) \succ_R y + x \succ_R R(y) \big)
        =
        R(x) \succ_{_R} R(y)
    \end{align*}
    The vertical alternative dendriform structure associated to the alternative quadri-algebra constructed with $R$
    coincides with the dendriform structure constructed from the associated alternative algebra $(A,\ast)$ and $R$:
    \begin{align*}
        x \wedge_R y
        &=
        x {\nearrow_{_R}} y + x {\nwarrow_{_R}} y
        =
        x \prec_R R(y) + x \succ_R R(y)
        =
        x \ast R(y) \\
        x \vee_R y
        &=
        x {\searrow_{_R}} y + x {\swarrow_{_R}} y
        =
        R(x) \succ_R y + R(x) \prec_R y
        =
        R(x) \ast y
    \end{align*}
\end{remark}

\begin{lemma} \label{lem:commRB}
    Let $(A,\ast)$ be an alternative algebra and $R_1$, $R_2$ two commuting Rota-Baxter operators on $A$.
    Then $R_2$ is a Rota-Baxter operator on the alternative dendriform dialgebra $(A,\prec_R,\succ_R)$.
\end{lemma}

\begin{proof}
    \begin{align*}
        R_2(x) \succ_{R_1} R_2(y)
        &=
        R_1( R_2(x) ) \ast R_2(y)
        =
        R_2( R_1(x) ) \ast R_2(y) \\
        &=
        R_2 \big( R_2(R_1(x)) \ast y + R_1(x) \ast R_2(y) \big) \\
        &=
        R_2 \big( R_1(R_2(x)) \ast y + R_1(x) \ast R_2(y) \big) \\
        &=
        R_2 \big( R_2(x) \succ_{R_1} y + x \succ_{R_1} R_2(y) \big)
    \end{align*}
    Similarly, $R_2(x) \prec_{R_1} R_2(y) = R_2 \big( R_2(x) \prec_{R_1} y + x \prec_{R_1} R_2(y) \big)$.
\end{proof}

\begin{proposition}
    Let $(A,\ast)$ be an alternative algebra and $R_1$, $R_2$ two commuting Rota-Baxter operators on $A$.
    Then $A$ is an alternative quadri-algebra with the operations defined by
    \begin{align*}
        x {\nearrow} y = R_1(x) \ast R_2(y) \qquad \qquad
        x {\searrow} y = R_1(R_2(x)) \ast y \\
        x {\swarrow} y = R_2(x) \ast R_1(y) \qquad \qquad
        x {\nwarrow} y = x \ast R_1(R_2(y))
    \end{align*}
\end{proposition}

\begin{proof}
    Use the construction in Prop.~\ref{prop:dendritoquadri} with $R_1$ and $(A,\prec_{R_2},\succ_{R_2})$.
\end{proof}


\section{M-dendriform algebras}

In this section we define M-dendriform algebras and give different constructions for them
similarly to what we did for pre-Malcev algebras in Section \ref{sec:preMalcev}.

\begin{definition}
    A \emph{M-dendriform algebra} is a vector space $M$ endowed with two bilinear products
    $\blacktriangleright$ and $\blacktriangleleft$ satisfying the identities
    \begin{align*}
        &
        MD1(x,y,z,t) =
        ((x {\mathsmaller{\mathsmaller{\blacktriangleright}}} y) {\mathsmaller{\mathsmaller{\blacktriangleright}}} z) {\mathsmaller{\mathsmaller{\blacktriangleright}}} t
        - ((y {\mathsmaller{\mathsmaller{\blacktriangleleft}}} x) {\mathsmaller{\mathsmaller{\blacktriangleright}}} z) {\mathsmaller{\mathsmaller{\blacktriangleright}}} t
        - (z {\mathsmaller{\mathsmaller{\blacktriangleleft}}} (x {\mathsmaller{\mathsmaller{\blacktriangleright}}} y)) {\mathsmaller{\mathsmaller{\blacktriangleright}}} t
        + (z {\mathsmaller{\mathsmaller{\blacktriangleleft}}} (y {\mathsmaller{\mathsmaller{\blacktriangleleft}}} x)) {\mathsmaller{\mathsmaller{\blacktriangleright}}} t
        \\
        &\quad
        - x {\mathsmaller{\mathsmaller{\blacktriangleright}}} (y {\mathsmaller{\mathsmaller{\blacktriangleleft}}} (z {\mathsmaller{\mathsmaller{\blacktriangleleft}}} t))
        - x {\mathsmaller{\mathsmaller{\blacktriangleright}}} (y {\mathsmaller{\mathsmaller{\blacktriangleright}}} (z {\mathsmaller{\mathsmaller{\blacktriangleleft}}} t))
        - x {\mathsmaller{\mathsmaller{\blacktriangleright}}} (y {\mathsmaller{\mathsmaller{\blacktriangleleft}}} (z {\mathsmaller{\mathsmaller{\blacktriangleright}}} t))
        - x {\mathsmaller{\mathsmaller{\blacktriangleright}}} (y {\mathsmaller{\mathsmaller{\blacktriangleright}}} (z {\mathsmaller{\mathsmaller{\blacktriangleright}}} t))
        + z {\mathsmaller{\mathsmaller{\blacktriangleleft}}} (x {\mathsmaller{\mathsmaller{\blacktriangleright}}} (y {\mathsmaller{\mathsmaller{\blacktriangleleft}}} t))
        \\
        &\quad
        + z {\mathsmaller{\mathsmaller{\blacktriangleleft}}} (x {\mathsmaller{\mathsmaller{\blacktriangleright}}} (y {\mathsmaller{\mathsmaller{\blacktriangleright}}} t))
        + (y {\mathsmaller{\mathsmaller{\blacktriangleleft}}} z) {\mathsmaller{\mathsmaller{\blacktriangleleft}}} (x {\mathsmaller{\mathsmaller{\blacktriangleright}}} t)
        + (y {\mathsmaller{\mathsmaller{\blacktriangleright}}} z) {\mathsmaller{\mathsmaller{\blacktriangleleft}}} (x {\mathsmaller{\mathsmaller{\blacktriangleright}}} t)
        - (z {\mathsmaller{\mathsmaller{\blacktriangleleft}}} y) {\mathsmaller{\mathsmaller{\blacktriangleleft}}} (x {\mathsmaller{\mathsmaller{\blacktriangleright}}} t)
        - (z {\mathsmaller{\mathsmaller{\blacktriangleright}}} y) {\mathsmaller{\mathsmaller{\blacktriangleleft}}} (x {\mathsmaller{\mathsmaller{\blacktriangleright}}} t)
        \\
        &\quad
        - y {\mathsmaller{\mathsmaller{\blacktriangleleft}}} ((z {\mathsmaller{\mathsmaller{\blacktriangleleft}}} x) {\mathsmaller{\mathsmaller{\blacktriangleright}}} t)
        + y {\mathsmaller{\mathsmaller{\blacktriangleleft}}} ((x {\mathsmaller{\mathsmaller{\blacktriangleright}}} z) {\mathsmaller{\mathsmaller{\blacktriangleright}}} t)
        = 0
        \\[2pt]
        &
        MD2(x,y,z,t) =
        ((x {\mathsmaller{\mathsmaller{\blacktriangleleft}}} y) {\mathsmaller{\mathsmaller{\blacktriangleright}}} z) {\mathsmaller{\mathsmaller{\blacktriangleright}}} t
        - ((y {\mathsmaller{\mathsmaller{\blacktriangleright}}} x) {\mathsmaller{\mathsmaller{\blacktriangleright}}} z) {\mathsmaller{\mathsmaller{\blacktriangleright}}} t
        - (z {\mathsmaller{\mathsmaller{\blacktriangleleft}}} (x {\mathsmaller{\mathsmaller{\blacktriangleleft}}} y)) {\mathsmaller{\mathsmaller{\blacktriangleright}}} t
        + (z {\mathsmaller{\mathsmaller{\blacktriangleleft}}} (y {\mathsmaller{\mathsmaller{\blacktriangleright}}} x)) {\mathsmaller{\mathsmaller{\blacktriangleright}}} t
        \\
        &\quad
        - x {\mathsmaller{\mathsmaller{\blacktriangleleft}}} (y {\mathsmaller{\mathsmaller{\blacktriangleright}}} (z {\mathsmaller{\mathsmaller{\blacktriangleleft}}} t))
        - x {\mathsmaller{\mathsmaller{\blacktriangleleft}}} (y {\mathsmaller{\mathsmaller{\blacktriangleright}}} (z {\mathsmaller{\mathsmaller{\blacktriangleright}}} t))
        + z {\mathsmaller{\mathsmaller{\blacktriangleleft}}} (x {\mathsmaller{\mathsmaller{\blacktriangleleft}}} (y {\mathsmaller{\mathsmaller{\blacktriangleright}}} t))
        + (y {\mathsmaller{\mathsmaller{\blacktriangleright}}} z) {\mathsmaller{\mathsmaller{\blacktriangleright}}} (x {\mathsmaller{\mathsmaller{\blacktriangleleft}}} t)
        + (y {\mathsmaller{\mathsmaller{\blacktriangleright}}} z) {\mathsmaller{\mathsmaller{\blacktriangleright}}} (x {\mathsmaller{\mathsmaller{\blacktriangleright}}} t)
        \\
        &\quad
        - (z {\mathsmaller{\mathsmaller{\blacktriangleleft}}} y) {\mathsmaller{\mathsmaller{\blacktriangleright}}} (x {\mathsmaller{\mathsmaller{\blacktriangleleft}}} t)
        - (z {\mathsmaller{\mathsmaller{\blacktriangleleft}}} y) {\mathsmaller{\mathsmaller{\blacktriangleright}}} (x {\mathsmaller{\mathsmaller{\blacktriangleright}}} t)
        - y {\mathsmaller{\mathsmaller{\blacktriangleright}}} ((z {\mathsmaller{\mathsmaller{\blacktriangleleft}}} x) {\mathsmaller{\mathsmaller{\blacktriangleleft}}} t)
        - y {\mathsmaller{\mathsmaller{\blacktriangleright}}} ((z {\mathsmaller{\mathsmaller{\blacktriangleleft}}} x) {\mathsmaller{\mathsmaller{\blacktriangleright}}} t)
        - y {\mathsmaller{\mathsmaller{\blacktriangleright}}} ((z {\mathsmaller{\mathsmaller{\blacktriangleright}}} x) {\mathsmaller{\mathsmaller{\blacktriangleleft}}} t)
        \\
        &\quad
        - y {\mathsmaller{\mathsmaller{\blacktriangleright}}} ((z {\mathsmaller{\mathsmaller{\blacktriangleright}}} x) {\mathsmaller{\mathsmaller{\blacktriangleright}}} t)
        + y {\mathsmaller{\mathsmaller{\blacktriangleright}}} ((x {\mathsmaller{\mathsmaller{\blacktriangleleft}}} z) {\mathsmaller{\mathsmaller{\blacktriangleleft}}} t)
        + y {\mathsmaller{\mathsmaller{\blacktriangleright}}} ((x {\mathsmaller{\mathsmaller{\blacktriangleleft}}} z) {\mathsmaller{\mathsmaller{\blacktriangleright}}} t)
        + y {\mathsmaller{\mathsmaller{\blacktriangleright}}} ((x {\mathsmaller{\mathsmaller{\blacktriangleright}}} z) {\mathsmaller{\mathsmaller{\blacktriangleleft}}} t)
        + y {\mathsmaller{\mathsmaller{\blacktriangleright}}} ((x {\mathsmaller{\mathsmaller{\blacktriangleright}}} z) {\mathsmaller{\mathsmaller{\blacktriangleright}}} t)
        = 0
        \\[2pt]
        &
        MD3(x,y,z,t)=
        ((x {\mathsmaller{\mathsmaller{\blacktriangleleft}}} y) {\mathsmaller{\mathsmaller{\blacktriangleleft}}} z) {\mathsmaller{\mathsmaller{\blacktriangleright}}} t
        + ((x {\mathsmaller{\mathsmaller{\blacktriangleright}}} y) {\mathsmaller{\mathsmaller{\blacktriangleleft}}} z) {\mathsmaller{\mathsmaller{\blacktriangleright}}} t
        - ((y {\mathsmaller{\mathsmaller{\blacktriangleleft}}} x) {\mathsmaller{\mathsmaller{\blacktriangleleft}}} z) {\mathsmaller{\mathsmaller{\blacktriangleright}}} t
        - ((y {\mathsmaller{\mathsmaller{\blacktriangleright}}} x) {\mathsmaller{\mathsmaller{\blacktriangleleft}}} z) {\mathsmaller{\mathsmaller{\blacktriangleright}}} t
        \\
        &\quad
        - (z {\mathsmaller{\mathsmaller{\blacktriangleright}}} (x {\mathsmaller{\mathsmaller{\blacktriangleleft}}} y)) {\mathsmaller{\mathsmaller{\blacktriangleright}}} t
        - (z {\mathsmaller{\mathsmaller{\blacktriangleright}}} (x {\mathsmaller{\mathsmaller{\blacktriangleright}}} y)) {\mathsmaller{\mathsmaller{\blacktriangleright}}} t
        + (z {\mathsmaller{\mathsmaller{\blacktriangleright}}} (y {\mathsmaller{\mathsmaller{\blacktriangleleft}}} x)) {\mathsmaller{\mathsmaller{\blacktriangleright}}} t
        + (z {\mathsmaller{\mathsmaller{\blacktriangleright}}} (y {\mathsmaller{\mathsmaller{\blacktriangleright}}} x)) {\mathsmaller{\mathsmaller{\blacktriangleright}}} t
        - x {\mathsmaller{\mathsmaller{\blacktriangleleft}}} (y {\mathsmaller{\mathsmaller{\blacktriangleleft}}} (z {\mathsmaller{\mathsmaller{\blacktriangleright}}} t))
        \\
        &\quad
        + z {\mathsmaller{\mathsmaller{\blacktriangleright}}} (x {\mathsmaller{\mathsmaller{\blacktriangleleft}}} (y {\mathsmaller{\mathsmaller{\blacktriangleleft}}} t))
        + z {\mathsmaller{\mathsmaller{\blacktriangleright}}} (x {\mathsmaller{\mathsmaller{\blacktriangleright}}} (y {\mathsmaller{\mathsmaller{\blacktriangleleft}}} t))
        + z {\mathsmaller{\mathsmaller{\blacktriangleright}}} (x {\mathsmaller{\mathsmaller{\blacktriangleleft}}} (y {\mathsmaller{\mathsmaller{\blacktriangleright}}} t))
        + z {\mathsmaller{\mathsmaller{\blacktriangleright}}} (x {\mathsmaller{\mathsmaller{\blacktriangleright}}} (y {\mathsmaller{\mathsmaller{\blacktriangleright}}} t))
        + (y {\mathsmaller{\mathsmaller{\blacktriangleleft}}} z) {\mathsmaller{\mathsmaller{\blacktriangleright}}} (x {\mathsmaller{\mathsmaller{\blacktriangleleft}}} t)
        \\
        &\quad
        + (y {\mathsmaller{\mathsmaller{\blacktriangleleft}}} z) {\mathsmaller{\mathsmaller{\blacktriangleright}}} (x {\mathsmaller{\mathsmaller{\blacktriangleright}}} t)
        - (z {\mathsmaller{\mathsmaller{\blacktriangleright}}} y) {\mathsmaller{\mathsmaller{\blacktriangleright}}} (x {\mathsmaller{\mathsmaller{\blacktriangleleft}}} t)
        - (z {\mathsmaller{\mathsmaller{\blacktriangleright}}} y) {\mathsmaller{\mathsmaller{\blacktriangleright}}} (x {\mathsmaller{\mathsmaller{\blacktriangleright}}} t)
        - y {\mathsmaller{\mathsmaller{\blacktriangleleft}}} ((z {\mathsmaller{\mathsmaller{\blacktriangleright}}} x) {\mathsmaller{\mathsmaller{\blacktriangleright}}} t)
        + y {\mathsmaller{\mathsmaller{\blacktriangleleft}}} ((x {\mathsmaller{\mathsmaller{\blacktriangleleft}}} z) {\mathsmaller{\mathsmaller{\blacktriangleright}}} t)
        = 0
        \\[2pt]
        &
        MD4(x,y,z,t) =
        ((x {\mathsmaller{\mathsmaller{\blacktriangleleft}}} y) {\mathsmaller{\mathsmaller{\blacktriangleleft}}} z) {\mathsmaller{\mathsmaller{\blacktriangleleft}}} t
        + ((x {\mathsmaller{\mathsmaller{\blacktriangleleft}}} y) {\mathsmaller{\mathsmaller{\blacktriangleright}}} z) {\mathsmaller{\mathsmaller{\blacktriangleleft}}} t
        + ((x {\mathsmaller{\mathsmaller{\blacktriangleright}}} y) {\mathsmaller{\mathsmaller{\blacktriangleleft}}} z) {\mathsmaller{\mathsmaller{\blacktriangleleft}}} t
        + ((x {\mathsmaller{\mathsmaller{\blacktriangleright}}} y) {\mathsmaller{\mathsmaller{\blacktriangleright}}} z) {\mathsmaller{\mathsmaller{\blacktriangleleft}}} t
        \\
        &\quad
        - ((y {\mathsmaller{\mathsmaller{\blacktriangleleft}}} x) {\mathsmaller{\mathsmaller{\blacktriangleleft}}} z) {\mathsmaller{\mathsmaller{\blacktriangleleft}}} t
        - ((y {\mathsmaller{\mathsmaller{\blacktriangleleft}}} x) {\mathsmaller{\mathsmaller{\blacktriangleright}}} z) {\mathsmaller{\mathsmaller{\blacktriangleleft}}} t
        - ((y {\mathsmaller{\mathsmaller{\blacktriangleright}}} x) {\mathsmaller{\mathsmaller{\blacktriangleleft}}} z) {\mathsmaller{\mathsmaller{\blacktriangleleft}}} t
        - ((y {\mathsmaller{\mathsmaller{\blacktriangleright}}} x) {\mathsmaller{\mathsmaller{\blacktriangleright}}} z) {\mathsmaller{\mathsmaller{\blacktriangleleft}}} t
        - (z {\mathsmaller{\mathsmaller{\blacktriangleleft}}} (x {\mathsmaller{\mathsmaller{\blacktriangleleft}}} y)) {\mathsmaller{\mathsmaller{\blacktriangleleft}}} t
        \\
        &\quad
        - (z {\mathsmaller{\mathsmaller{\blacktriangleright}}} (x {\mathsmaller{\mathsmaller{\blacktriangleleft}}} y)) {\mathsmaller{\mathsmaller{\blacktriangleleft}}} t
        - (z {\mathsmaller{\mathsmaller{\blacktriangleleft}}} (x {\mathsmaller{\mathsmaller{\blacktriangleright}}} y)) {\mathsmaller{\mathsmaller{\blacktriangleleft}}} t
        - (z {\mathsmaller{\mathsmaller{\blacktriangleright}}} (x {\mathsmaller{\mathsmaller{\blacktriangleright}}} y)) {\mathsmaller{\mathsmaller{\blacktriangleleft}}} t
        + (z {\mathsmaller{\mathsmaller{\blacktriangleleft}}} (y {\mathsmaller{\mathsmaller{\blacktriangleleft}}} x)) {\mathsmaller{\mathsmaller{\blacktriangleleft}}} t
        + (z {\mathsmaller{\mathsmaller{\blacktriangleright}}} (y {\mathsmaller{\mathsmaller{\blacktriangleleft}}} x)) {\mathsmaller{\mathsmaller{\blacktriangleleft}}} t
        \\
        &\quad
        + (z {\mathsmaller{\mathsmaller{\blacktriangleleft}}} (y {\mathsmaller{\mathsmaller{\blacktriangleright}}} x)) {\mathsmaller{\mathsmaller{\blacktriangleleft}}} t
        + (z {\mathsmaller{\mathsmaller{\blacktriangleright}}} (y {\mathsmaller{\mathsmaller{\blacktriangleright}}} x)) {\mathsmaller{\mathsmaller{\blacktriangleleft}}} t
        - x {\mathsmaller{\mathsmaller{\blacktriangleleft}}} (y {\mathsmaller{\mathsmaller{\blacktriangleleft}}} (z {\mathsmaller{\mathsmaller{\blacktriangleleft}}} t))
        + z {\mathsmaller{\mathsmaller{\blacktriangleleft}}} (x {\mathsmaller{\mathsmaller{\blacktriangleleft}}} (y {\mathsmaller{\mathsmaller{\blacktriangleleft}}} t))
        + (y {\mathsmaller{\mathsmaller{\blacktriangleleft}}} z) {\mathsmaller{\mathsmaller{\blacktriangleleft}}} (x {\mathsmaller{\mathsmaller{\blacktriangleleft}}} t)
        \\
        &\quad
        + (y {\mathsmaller{\mathsmaller{\blacktriangleright}}} z) {\mathsmaller{\mathsmaller{\blacktriangleleft}}} (x {\mathsmaller{\mathsmaller{\blacktriangleleft}}} t)
        - (z {\mathsmaller{\mathsmaller{\blacktriangleleft}}} y) {\mathsmaller{\mathsmaller{\blacktriangleleft}}} (x {\mathsmaller{\mathsmaller{\blacktriangleleft}}} t)
        - (z {\mathsmaller{\mathsmaller{\blacktriangleright}}} y) {\mathsmaller{\mathsmaller{\blacktriangleleft}}} (x {\mathsmaller{\mathsmaller{\blacktriangleleft}}} t)
        - y {\mathsmaller{\mathsmaller{\blacktriangleleft}}} ((z {\mathsmaller{\mathsmaller{\blacktriangleleft}}} x) {\mathsmaller{\mathsmaller{\blacktriangleleft}}} t)
        - y {\mathsmaller{\mathsmaller{\blacktriangleleft}}} ((z {\mathsmaller{\mathsmaller{\blacktriangleright}}} x) {\mathsmaller{\mathsmaller{\blacktriangleleft}}} t)
        \\
        &\quad
        + y {\mathsmaller{\mathsmaller{\blacktriangleleft}}} ((x {\mathsmaller{\mathsmaller{\blacktriangleleft}}} z) {\mathsmaller{\mathsmaller{\blacktriangleleft}}} t)
        + y {\mathsmaller{\mathsmaller{\blacktriangleleft}}} ((x {\mathsmaller{\mathsmaller{\blacktriangleright}}} z) {\mathsmaller{\mathsmaller{\blacktriangleleft}}} t)
        = 0
    \end{align*}
    for all $x,y,z,t \in M$.
\end{definition}

\begin{remark}
    This definition of pre-Malcev algebras comes from applying the splitting procedure described in \cite{BaiBellierGuoNi12}
    to pre-Malcev algebras.
    The stated identities are a minimal set of generators of the S4-module generated by the obtained identities
    (we applied the module generators algorithm detailed in \cite{BremnerMadariagaPeresi14}).
\end{remark}

\begin{proposition}\label{prop:M-dendriform}
    If $(M,\blacktriangleright,\blacktriangleleft)$ is a M-dendriform algebra then the product
    $x \cdot_h y = x \mathsmaller{\blacktriangleright} y + x \mathsmaller{\blacktriangleleft} y$
    (resp. $x \cdot_v y = x \mathsmaller{\blacktriangleright} y - y \mathsmaller{\blacktriangleleft} x$)
    defines a pre-Malcev algebra: the horizontal (resp. vertical) pre-Malcev algebra associated to $M$.
\end{proposition}

\begin{proof}
    Straightforward computations.
\end{proof}

M-dendriform algebras are closely related to bimodules for pre-Malcev algebras.

\begin{definition}
    Let $(A,\cdot)$ be a pre-Malcev algebra, $V$ be a vector space and $\ell,r \colon A \to \mathrm{gl}(V)$
    be two linear maps. Then $(V,\ell,r)$ is a \emph{bimodule} of $(A,\cdot)$ if the following hold:
    \begin{align*}
        &
          r_x r_y r_z
        - r_x r_y \ell_z
        - r_x \ell_y r_z
        + r_x \ell_y \ell_z
        - r_{z \cdot (y \cdot x)}
        + \ell_y r_{z \cdot x}
        + \ell_{z \cdot y} r_x
        - \ell_{y \cdot z} r_x
        \\
        & \quad
        - \ell_z r_x \ell_y
        + \ell_z r_x r_y
        = 0
        \\
        &
          r_x r_y \ell_z
        - r_x r_y r_z
        - r_x \ell_y \ell_z
        + r_x \ell_y r_z
        - \ell_z r_{y \cdot x}
        + \ell_y \ell_z r_x
        + r_{z \cdot x} r_y
        - r_{z \cdot x} \ell_y
        \\
        & \quad
        - r_{(y \cdot z) \cdot x}
        + r_{(z \cdot y) \cdot x}
        = 0
        \\
        &
          r_x \ell_{y \cdot z}
        - r_x \ell_{z \cdot y}
        - r_x r_{y \cdot z}
        + r_x r_{z \cdot y}
        - \ell_y \ell_z r_x
        + r_{y \cdot (z \cdot x)}
        + r_{y \cdot x} \ell_z
        - r_{y \cdot x} r_z
        \\
        & \quad
        - \ell_z r_x r_y
        + \ell_z r_x \ell_y
        = 0
        \\
        &
          \ell_{(x \cdot y) \cdot z}
        - \ell_{(y \cdot x) \cdot z}
        - \ell_{z \cdot (x \cdot y)}
        + \ell_{z \cdot (y \cdot x)}
        - \ell_x \ell_y \ell_z
        + \ell_z \ell_x \ell_y
        + \ell_{y \cdot z} \ell_x
        - \ell_{z \cdot y} \ell_x
        \\
        & \quad
        - \ell_y \ell_{z \cdot x}
        + \ell_y \ell_{x \cdot z}
        = 0
    \end{align*}
    for all $x,y,z \in A$.
\end{definition}

\begin{proposition}
    $(M,\blacktriangleright,\blacktriangleleft)$ is a M-dendriform algebra if and only if
    $(M,\cdot_h)$ (resp. $(M,\cdot_v)$) is a pre-Malcev algebra and
    $(M,L_\blacktriangleleft,R_\blacktriangleright)$ (resp. $(M,L_\blacktriangleright,{-}L_\blacktriangleright)$)
    is a bimodule for $(M,\cdot_h)$ (resp. $(M,\cdot_v)$).
\end{proposition}

\begin{proof}
    It follows from Prop.~\ref{prop:M-dendriform} above and the definition of representation of a pre-Malcev algebra.
\end{proof}

\begin{remark}
    Since Malcev algebras generalize Lie algebras, M-dendriform algebras generalize L-dendriform algebras.
    Recall that a \emph{L-den\-dri\-form algebra} \cite{BaiLiuNi10} is a vector space with two bilinear operations
    $\blacktriangleleft$, $\blacktriangleright$ satisfying
    \begin{align*}
        LD1(x,y,z)
        &=
          x \mathsmaller{\mathsmaller{\blacktriangleright}} (y \mathsmaller{\mathsmaller{\blacktriangleright}} z)
        - (x \mathsmaller{\mathsmaller{\blacktriangleright}} y) \mathsmaller{\mathsmaller{\blacktriangleright}} z
        - (x \mathsmaller{\mathsmaller{\blacktriangleleft}} y) \mathsmaller{\mathsmaller{\blacktriangleright}} z
        - y \mathsmaller{\mathsmaller{\blacktriangleright}} (x \mathsmaller{\mathsmaller{\blacktriangleright}} z)
        + (y\mathsmaller{\mathsmaller{\blacktriangleleft}}x) \mathsmaller{\mathsmaller{\blacktriangleright}} z
        + (y \mathsmaller{\mathsmaller{\blacktriangleright}} x) \mathsmaller{\mathsmaller{\blacktriangleright}} z
        \\
        LD2(x,y,z)
        &=
          x \mathsmaller{\mathsmaller{\blacktriangleright}} (y \mathsmaller{\mathsmaller{\blacktriangleright}} z)
        - (x \mathsmaller{\mathsmaller{\blacktriangleright}} y) \mathsmaller{\mathsmaller{\blacktriangleleft}} z
        - y \mathsmaller{\mathsmaller{\blacktriangleleft}} (x \mathsmaller{\mathsmaller{\blacktriangleright}} z)
        - y \mathsmaller{\mathsmaller{\blacktriangleleft}} (x \mathsmaller{\mathsmaller{\blacktriangleleft}} z)
        + (y \mathsmaller{\mathsmaller{\blacktriangleleft}} x) \mathsmaller{\mathsmaller{\blacktriangleleft}} z.
    \end{align*}
\end{remark}

\begin{remark}
    M-dendriform algebras do not generalize alternative dendriform algebras:
    the corresponding sets of defining identities do not generate the same $S_4$-module.
\end{remark}

Examples of M-dendriform algebras can be constructed from pre-Malcev algebras with Rota-Baxter operators of weight zero.
Recall that a \emph{Rota-Baxter operator of weight zero} on a pre-Malcev algebra $(M,- \cdot -)$ is a linear map $R \colon M \to M$ such that
\[
R(x) \cdot R(y) = R\big( R(x) \cdot y + x \cdot R(y) \big) \qquad \forall x,y \in M.
\]

\begin{proposition}
    If $R \colon M \to M$ is a Rota-Baxter operator on a pre-Malcev algebra $(M,- \cdot -)$ then
    there exists a M-dendriform algebraic structure on $M$ given by
    \[
        x \mathsmaller{\blacktriangleright} y = R(x) \cdot y, \quad
        x \mathsmaller{\blacktriangleleft}  y = x \cdot R(y) \qquad \forall x,y \in M.
    \]
\end{proposition}

\begin{proof}
    Straightforward computations similar to these in Prop.~\ref{prop:Malcev-RB}.
\end{proof}

\begin{remark}
    We also computed the disuccessor of M-dendriform algebras (i.e. what we would call \emph{M-quadri-algebras}),
    obtaining 15 non-redundant defining identities which are too messy to be displayed here. \\
    M-quadrialgebras can be obtained from Malcev algebras and two commuting Rota-Baxter operators
    using \cite[Prop.~3.6]{GubarevKolesnikov13} and the analogous of Lemma \ref{lem:commRB}.
    A natural question arises here: can all M-quadrialgebras be obtained this way? \\
    It was also checked that M-quadri-algebras do not generalize alternative quadri-algebras:
    the $S_4$-module generated by these 15 identities is not contained in the $S_4$-module of the consequences
    in degree 4 of the identities in Definition \ref{def:altquad}.
    Compare this result with what happens in the Jordan case (see Remark~\ref{rem:J-quadri}).
\end{remark}

M-dendriform algebras are related to alternative quadri-algebras in the same way L-dendriform algebras are related to quadri-algebras.

\begin{proposition} \label{prop:quadritoM}
    Let $(A,\nearrow,\searrow,\swarrow,\nwarrow)$ an alternative quadri-algebra. Then the operations
    \[
        x \mathsmaller{\blacktriangleleft} y = x \nearrow y - y \swarrow x
        \qquad
        x \mathsmaller{\blacktriangleright} y = x \searrow y - y \nwarrow x
    \]
    define a M-dendriform structure in $A$. Moreover, all the identities of degree 4 satisfied by these operations
    in the free alternative quadri-algebra are consequences of $MDi(x,y,z,t)$, $i=1 \dots 4$.
\end{proposition}

\begin{proof}
    Similar to the proof of Prop.~\ref{prop:computations}.
    We consider the expansion map $E_4$ from the free nonassociative algebra with four bilinear operations $\mathsf{BBBB}$
    to the free algebra M-dendriform algebra $\mathsf{BB}$.
    In this case there are 40 $\mathsf{BB}$-association types in degree 4, so 960 multilinear $\mathsf{BB}$-monomials
    There are 320 $\mathsf{BBBB}$-association types in degree 4, so 7680 multilinear $\mathsf{BBBB}$-monomials.
    The matrix representing $E_4$ has size $5280 \times 8640$.
    There are 80 rows in its row canonical form whose leading 1s are in the right blocks.
    The corresponding identities linearly generate the same $S_4$-module as the defining identities for M-dendriform algebras.
    We note that the products $\blacktriangleleft$ and $\blacktriangleright$ defined above
    satisfy no identities in degree 3 in the free alternative quadri-algebra.
\end{proof}


\section{Relations between the alternative and Jordan structures}

Analogously to what happens in the associative case, the anticommutator or Jordan product $x \circ y = x \ast y + y \ast x$
in an alternative algebra $(A,\ast)$ defines a Jordan algebra structure in $A$.
Jordan algebras were first introduced by Jordan \cite{Jordan32} to formalize the notion of an algebra of observables in quantum mechanics.
In this section we show that the relations between Jordan and associative structures in Figure~\ref{fig:classic}
can be naturally generalized to relations between the same Jordan the corresponding alternative structures, as
displayed in Figure~\ref{fig:alternative}.

\begin{definition}[Albert, \cite{Albert46}]
    A \emph{Jordan algebra} is a vector space $A$ with a commutative bilinear multiplication $\circ$
    satisfying the Jordan identity
    \[
        (x \circ y) \circ (x \circ x) = x \circ (y \circ (x \circ x))
    \]
    for all $x,y \in A$.
\end{definition}

Hou, Ni and Bai \cite{HouNiBai13} introduced pre-Jordan algebras in relation to the Jordan Yang-Baxter equation and
considered them analogues of pre-Lie algebras in terms of $\mathcal{O}$-operators, that is,
the algebraic structure satisfying that its anticommutator is a Jordan algebra and its left multiplication operators
give a representation of this Jordan algebra.
They also study splitting of associativity in the Jordan setting.

\begin{definition}[Hou, Ni and Bai \cite{HouNiBai13}]
    A \emph{pre-Jordan algebra} is a vector space $A$ with a bilinear product $\cdot \colon A \to A$ such that
    \begin{align*}
    &
      (x \circ y) \cdot (z \cdot t) + (y \circ z) \cdot (x \cdot t) + (z \circ x) \cdot (y \cdot t)
    - z \cdot ((x \circ y) \cdot t) - x \cdot ((y \circ z) \cdot t)
    \\
    & \quad - y \cdot ((z \circ x) \cdot t) = 0
    \\[2pt]
    &
      x \cdot (y \cdot (z \cdot t)) + z \cdot (y \cdot (x \cdot t)) + ((x \circ z) \circ y) \cdot t
    - z \cdot ((x \circ y) \cdot t) - x \cdot ((y \circ z) \cdot t)
    \\
    & \quad - y \cdot ((z \circ x) \cdot t) = 0
    \end{align*}
    for all $x,y,z,t \in A$, where $x \circ y = x \cdot y + y \cdot x$.
\end{definition}

Pre-Jordan algebras can be obtained from both dendriform and alternative dendriform dialgebras using the
dendriform anticommutator.

\begin{proposition}
    Let $(A,\prec,\succ)$ an alternative dendriform dialgebra. Then the dendriform anticommutator
    \[
        x \cdot y = x \succ y + y \prec x
    \]
    defines a pre-Jordan structure in $A$.
\end{proposition}

\begin{proof}
    It follows from \cite[Proposition 2.29(a)]{BaiBellierGuoNi12}.
\end{proof}

Bai and Hou \cite{BaiHou12} introduced J-dendriform algebras as the Jordan algebraic analogue of dendriform dialgebras
in the sense that the anticommutator of the sum of the two operations is a Jordan algebra.
Also, they are related to pre-Jordan algebras in the same way that pre-Jordan algebras are related to Jordan algebras.
They showed that a Rota-Baxter operator on a pre-Jordan algebra or two commuting Rota-Baxter operators on a
Jordan algebra give a J-dendriform algebra.
They also proved the relation between J-dendriform algebras and quadri-algebras.

\begin{definition}[Bai and Hou \cite{BaiHou12}]\label{def:J-dendriform}
    A \emph{J-dendriform algebra} is a vector space $A$ with a two bilinear products
    $\blacktriangleleft, \blacktriangleright \colon A \to A$ such that
    \begin{align*}
    &
      (x \circ y) \mathsmaller{\blacktriangleright} (z \mathsmaller{\blacktriangleright} t )
    + (y \circ z) \mathsmaller{\blacktriangleright} (x \mathsmaller{\blacktriangleright} t)
    + (z \circ x) \mathsmaller{\blacktriangleright} (y \mathsmaller{\blacktriangleright} t)
    - x \mathsmaller{\blacktriangleright} ((y \circ z) \mathsmaller{\blacktriangleright} t)
    - y \mathsmaller{\blacktriangleright} ((z \circ x) \mathsmaller{\blacktriangleright} t)
    \\
    & \quad
    - z \mathsmaller{\blacktriangleright} ((x \circ y) \mathsmaller{\blacktriangleright} t)
    = 0
    \\[2pt]
    &
      (x \circ y) \mathsmaller{\blacktriangleright} (z \mathsmaller{\blacktriangleright} t )
    + (y \circ z) \mathsmaller{\blacktriangleright} (x \mathsmaller{\blacktriangleright} t)
    + (z \circ x) \mathsmaller{\blacktriangleright} (y \mathsmaller{\blacktriangleright} t)
    - x \mathsmaller{\blacktriangleright} (y \mathsmaller{\blacktriangleright} (z \mathsmaller{\blacktriangleright} t))
    - z \mathsmaller{\blacktriangleright} ( y \mathsmaller{\blacktriangleright} (x \mathsmaller{\blacktriangleright} t))
    \\
    & \quad
    - (y \circ (z \circ x)) \mathsmaller{\blacktriangleright} t
    \\[2pt]
    &
      (x \circ y) \mathsmaller{\blacktriangleright} (z \mathsmaller{\blacktriangleleft} t)
    + (x \cdot z) \mathsmaller{\blacktriangleleft} (y \diamond t)
    + (y \cdot z) \mathsmaller{\blacktriangleleft} (x \diamond t)
    - x \mathsmaller{\blacktriangleright} (z \mathsmaller{\blacktriangleleft} (y \diamond t))
    - y \mathsmaller{\blacktriangleright} (z \mathsmaller{\blacktriangleleft} (x \diamond t))
    \\
    & \quad
    - ((x \circ y) \cdot z) \mathsmaller{\blacktriangleleft} t
    \\[2pt]
    &
      (z \cdot y) \mathsmaller{\blacktriangleleft} (x \diamond t)
    + (x \cdot y) \mathsmaller{\blacktriangleleft} (z \diamond t)
    + (x \circ z) \mathsmaller{\blacktriangleright} (y \mathsmaller{\blacktriangleleft} t)
    - x \mathsmaller{\blacktriangleright} ((z \cdot y) \mathsmaller{\blacktriangleleft} t)
    - z \mathsmaller{\blacktriangleright} ((x \cdot y) \mathsmaller{\blacktriangleleft} t)
    \\
    & \quad
    - y \mathsmaller{\blacktriangleleft} ((x \circ z) \diamond t)
    \\[2pt]
    &
      (x \circ y) \mathsmaller{\blacktriangleright} (z \mathsmaller{\blacktriangleleft} t)
    + (x \cdot z) \mathsmaller{\blacktriangleleft} (y \diamond t)
    + (y \cdot z) \mathsmaller{\blacktriangleleft} (x \diamond t)
    - x \mathsmaller{\blacktriangleright} (y \mathsmaller{\blacktriangleright} (z \mathsmaller{\blacktriangleleft} t))
    - z \mathsmaller{\blacktriangleleft} (y \diamond (x \diamond t))
    \\
    & \quad
    - (y \cdot (x \cdot z)) \mathsmaller{\blacktriangleleft} t
    \end{align*}
    for all $x,y,z,t \in A$, where
    \begin{align*}
        &
        \qquad\qquad x \cdot y = x \mathsmaller{\blacktriangleright} y + y \mathsmaller{\blacktriangleleft} x \qquad x \diamond y
        =
        x \mathsmaller{\blacktriangleright} y + x \mathsmaller{\blacktriangleleft} y
        \\
        &
        x \circ y = x \cdot y + y \cdot x = x \diamond y + y \diamond x
        =
        x \mathsmaller{\blacktriangleright} y + x \mathsmaller{\blacktriangleleft} y
        +
        y \mathsmaller{\blacktriangleright} x + y \mathsmaller{\blacktriangleleft} x.
    \end{align*}
\end{definition}

\begin{remark}
    J-dendriform algebras do not generalize alternative dendriform algebras:
    the corresponding sets of defining identities do not generate the same $S_4$-module.
\end{remark}

J-dendriform algebras can be obtained from alternative quadri-algebras analogously to how they can be obtained from
quadri-algebras.

\begin{proposition}
    Let $(A,\nearrow,\searrow,\swarrow,\nwarrow)$ an alternative quadri-algebra. Then the operations
    \[
        x \mathsmaller{\blacktriangleleft} y = x \nearrow y + y \swarrow x
        \qquad
        x \mathsmaller{\blacktriangleright} y = x \searrow y + y \nwarrow x
    \]
    define a J-dendriform structure in $A$. Moreover, all the identities of degree 4 satisfied by these operations
    in the free alternative quadri-algebra are consequences of the identities in Def.~\ref{def:J-dendriform}.
\end{proposition}

\begin{proof}
    It follows from \cite[Proposition 2.29(a)]{BaiBellierGuoNi12}.
\end{proof}

\begin{remark} \label{rem:J-quadri}
    We also computed the disuccessor of J-dendriform algebras (i.e. what we would call J-quadri-algebras),
    obtaining 15 defining identities which are too long to be displayed here.
    Interestingly enough, J-quadri-algebras generalize alternative quadri-algebras,
    the $S_4$-module generated by these 15 identities is contained in the $S_4$-module of the consequences
    in degree 4 of the identities in Definition \ref{def:altquad}.
\end{remark}


\section{Conclusions and further research}

The computations described in this paper can be extended to other structures using the trisuccessors defined in \cite{BaiBellierGuoNi12}.
This will result in the explicit definitions of tri-alternative dendriform algebras, alternative ennea-algebras, etc.,
generalizing the associative case.
We can also investigate what happens if we replace associative (or alternative) algebras by flexible algebras.
However, the interest of such computations is limited to the extent to it would enlighten other researchers.


\section*{Acknowledgements}

The author thanks Prof. M.R.~Bremner for suggesting this problem to her, for his scientific generosity,
for his constant help and support and for his useful comments and suggestions. \\

She also appreciates the referee for its useful comments and insights, which contributed to significantly improve this work.


\end{document}